\documentclass[11pt]{amsart}

\usepackage{latexsym,rawfonts}
\usepackage{amsfonts,amssymb}
\usepackage{amsmath,amsthm}

\usepackage[plainpages=false]{hyperref}
\usepackage{graphicx}
\usepackage{color}

\newtheorem{theorem}{Theorem}

\newtheorem{lemma}{Lemma}
\newtheorem{proposition}{Proposition}

\theoremstyle{definition}
\newtheorem{definition}{Definition}
\newtheorem{remark}{Remark}

\newcommand{\beq}{\begin{equation}}
\newcommand{\eeq}{\end{equation}}
\newcommand{\beqs}{\begin{eqnarray*}}
\newcommand{\eeqs}{\end{eqnarray*}}
\newcommand{\beqn}{\begin{eqnarray}}
\newcommand{\eeqn}{\end{eqnarray}}
\newcommand{\beqa}{\begin{array}}
\newcommand{\eeqa}{\end{array}}

\usepackage{amsmath}
\numberwithin{equation}{section}
\numberwithin{theorem}{section}
\numberwithin{lemma}{section}
\numberwithin{remark}{section}
\numberwithin{assumption}{section}
\numberwithin{definition}{section}
\numberwithin{fact}{section}
\numberwithin{question}{section}
\numberwithin{proposition}{section}

\begin{document}

\title{$C^{0}$ estimates for Hessian quotient equations on HKT manifolds}

\author{Li Chen}
\address{Faculty of Mathematics and Statistics, Hubei Key Laboratory of Applied
Mathematics, Hubei University, Wuhan 430062, P.R. China}
\email{chenli@hubu.edu.cn}


\date{}

\begin{abstract}
We show the $C^0$ estimate for solutions to
Hessian quotient equations on hyperK\"ahler with torsion manifolds without any additional assumption on its
hypercomplex structure by
a clever use of the cone condition directly.
\end{abstract}



\maketitle
\vskip4ex

\section{Introduction}

Let $(M, I, J, K)$ be a compact hyperK\"ahler with torsion, later abbreviated as HKT, manifold
and $g$ be the hyperhermitian metric of $M$. We denote by $\Omega$ the associated HKT
form with respect to $g$. Fix a q-real smooth closed $(2, 0)$-form $\Omega_0$ on $M$, we introduce a
new q-real $(2,0)$ form for any smooth function $u: M\rightarrow \mathbb{R}$
\begin{eqnarray*}
\Omega_u:=\Omega+\partial\partial_J u.
\end{eqnarray*}

We consider the following Hessian quotient equation on $(M, \Omega)$
\begin{eqnarray}\label{K-eq}
\Omega^{k}_{u}\wedge \Omega^{n-k}=e^F\Omega^{l}_{u}\wedge \Omega^{n-l}, \quad 0\leq l<k\leq n,
\end{eqnarray}
where $F$ is a real smooth functions on $M$.
For $k=n$ and $l=0$,
the equation \eqref{K-eq} is just the quaternionic Monge-Amp\`ere equation associated to
the quaternionic Calabi conjecture which was introduced by Alesker-Verbitsky \cite{Ale10}
formulated in analogy with the famous complex Calabi conjecture solved by Yau \cite{Yau78}.
Later, M. Verbitsky \cite{Ver09} has found
a geometric interpretation of this quaternionic Monge-Amp\`ere equation.

Let us now give an overview of the advances towards proving the quaternionic Calabi conjecture.
Until now, this conjecture was solved by Alesker in \cite{Ale13} on the HKT manifold with a flat
hyperK\"ahler metric, by Gentili-Vezzonion \cite{Gen22} $8$-dimensional $2$-step nilmanifolds $M$
with an abelian hypercomplex structure and by Dinew-Sroka \cite{Di21}
for compact hyperK\"ahler manifolds. But for compact HKT manifolds, this conjecture is still open.

The nature strategy to face this problem is, of course, to use the continuity method
for which a priori estimates are crucial. Here, the $C^0$ estimate is
a very important part in solving this conjecture. Alesker-Verbitsky \cite{Ale10}
obtained the $C^0$ estimate in the case when the canonical bundle is trivial by
repeating the classical Moser iteration method used by Yau in \cite{Yau78}.
This bound was shown to hold in \cite{Ale13-1} when the hypercomplex structure
is locally flat by using the method of B{\l}ocki from \cite{Blo05-1}.
However, the $C^0$ estimate can be established
in more general settings now. Alesker-Shelukhin \cite{Ale17} obtained the $C^0$ estimate
on a compact HKT manifold
without any additional assumption on its
HKT structure, following the scheme of B{\l}ocki \cite{Blo05-1}.
Recently, it is Sroka \cite{Sro21}, who provided a simpler proof, by
using an analogous approach of Cherrier \cite{Che87}, Tosatti-Weinkove \cite{Tos10}, or
Zhang \cite{Zhang17} on hermitian manifolds.

For $k=n$ and $l=n-1$,
the equation \eqref{K-eq} is analogous to the $J$-equation on K\"ahler manifolds  proposed
by Donaldson \cite{Don99} in the setting of moment maps.
The $J$-equation was solved by Song-Weinkove \cite{Song08} via $J$-flow after some progresses
made in \cite{Chen04, Wei04, Wei04D, Wei06}.

Recently, Zhang \cite{Zhang21} and Gentili \cite{Gen21} investigated
 a general class of fully non-linear equations on HKT manifolds $(M, I, J, K, g)$
\begin{eqnarray}\label{ZG}
G(\Omega_u)=e^F,
\end{eqnarray}
which includes the Hessian quotient equation \eqref{K-eq}.
They solved the equation \eqref{ZG} independently on closed HKT manifold
with a flat hyperK\"ahler metric under the existence of an admissible $\mathcal{C}$-subsolution by adapting the approach of Sz\'ekelyhidi \cite{Sze18}
to the hypercomplex setting. In particular,
the $C^0$ estimate for solution of the equation \eqref{ZG} was established in \cite{Zhang21, Gen21} by using
the Alexandroff-Bakelman-Pucci (ABP) method as in \cite{Sze18} which
generalized the scheme of B{\l}ocki \cite{Blo05-1}.

Thus, it is a very interesting problem to solve the
equation \eqref{ZG} on the compact KHT manifold without any additional assumption on its
HKT structure.  In this paper, we can make some progress on this problem. In details,
we can prove the $C^0$ estimate for solutions of the Hessian quotient equation \eqref{K-eq} with
on the compact KHT manifold without any additional assumption on its
HKT structure. To statement our main result, we need to introduce the cone condition which
is similar to the K\"ahler setting
\cite{Song08, Fang11, Guan15, Sun17}.
\begin{definition}
Let $\Omega_0$ be a q-real smooth closed $(2, 0)$-form on $M$,
we say that $\Omega_0$ satisfies the cone condition for the equation \eqref{K-eq} with respect with $\Omega$
if it satisfies
\begin{eqnarray}\label{cone}
k\Omega_{0}^{k-1}\wedge \Omega^{n-k}>le^F\Omega_{0}^{l-1}\wedge \Omega^{n-l}.
\end{eqnarray}
\end{definition}

It is easy to see that the cone condition \eqref{cone} is the necessary condition
for the solvability of \eqref{K-eq} in view of \eqref{Keq-1}, \eqref{initial-1} and \eqref{cone-local}.
We get the following result.
\begin{theorem}\label{Main}
Let $(M, I, J, K, \Omega)$ be a closed HKT manifold of dimension $n$,
and $\Omega_0 \in \Gamma_k$ be a
smooth q-real $(2, 0)$-form on $M$ with $\partial \Omega_0=0$. Assume $\Omega_0$ satisfies the
 cone condition \eqref{cone}, then
for any smooth solution $u$ of the Hessian quotient equation \eqref{K-eq} with $\sup_{M}u=0$,
the following bound holds
\begin{eqnarray*}
|u|_{C^0(M)}\leq C,
\end{eqnarray*}
where the constant $C$ depending only on the HKT structure and $F$.
\end{theorem}

The proof we present is strongly motivated by the work of Sun \cite{Sun17}
in which he derive $C^0$ estimate for solutions of Hessian quotient equations on K\"ahler manifolds
by a clever use of the cone condition directly. This in turn is based on an inequality obtained originally
by Cherrier in \cite{Che87}. The method emerged in the course of proving the $C^0$ estimate for the complex Monge-Amp\`ere
equation and Hessian equations on a compact hermitian manifold \cite{Tos10, Zhang17}.

\section{Preliminaries}
\subsection{HyperK\"ahler manifold with torsion}
Let us recall
the definition of HKT manifolds. A hypercomplex manifold is a smooth
manifold $M$ of real dimension $4n$ equipped with a triple
of complex structures $(I, J, K)$ satisfying the quaternionic relations
\begin{eqnarray*}
I\circ J=-J\circ I=K.
\end{eqnarray*}
A hypercomplex manifold $(M, I, J, K, g)$ with a Riemannian metric $g$ is called a hyperhermitian manifold if $g$ is
invariant under the three complex structures $(I, J, K)$, i.e.
\begin{eqnarray*}
g=g( \cdot I ,  \cdot I )=g(\cdot J, \cdot J)=g(\cdot K, \cdot K ).
\end{eqnarray*}
This action extends uniquely to the
right action of the algebra $\mathbb{H}$ of quaternions on $TX$.
A hypercomplex manifold admits the whole sphere of complex structures namely
\begin{eqnarray*}
S_{M}=\Big\{a I+b J+c K: a, b, c \in \mathbb{R}, \ a^2+b^2+c^2=1\Big\}.
\end{eqnarray*}
and a hyperhermitian metric $g$ is invariant under all of them.

For a given
$L \in S_M$ we denote the associated hermitian form by $\omega_L$, i.e.
$\omega_L=g(\cdot L, \cdot)$. Consider the following differential form
\begin{eqnarray*}
\Omega:=\omega_J-\sqrt{-1}\omega_K.
\end{eqnarray*}

\begin{definition}\label{HKT}
A hyperhermitian manifold $(M, I, J, K, g)$ is called HKT if
\begin{eqnarray*}
\partial \Omega=0,
\end{eqnarray*}
where $\partial$ in the whole paper is the differential operator with respect to
$I$.  To be more precise,
\begin{eqnarray*}
\partial=\frac{1}{2}(d+\sqrt{-1}I\circ d\circ I)
\end{eqnarray*}
and $d$ is the standard exterior differential operator on $M$.
\end{definition}

HKT manifolds were introduced in the physical literature by Howe and
Papadopoulos \cite{How96}. For the mathematical treatment see Grantcharov-Poon \cite{Gra00}
and Verbitsky \cite{Ver02}. The original definition of HKT-metrics in \cite{How96} was different but equivalent
to Definition \ref{HKT} (the latter was given in \cite{Ver02}). The classical hyperK\"ahler metrics (i.e. Riemannian metrics with the
holonomy of the Levi-Civita connection contained in the group $Sp(n)$) form a subclass
of HKT-metrics. It is well known that a hyperhermitian metric $g$ is hyperK\"ahler if and
only if the form $\Omega$ is closed, or equivalently $\partial\Omega=0$ and $\overline{\partial} \Omega=0$.

Let $(M, I, J, K)$ be a hypercomplex manifold. Let us denote by $\Lambda^{p,q}_{I}(M)$ the vector bundle
of differential forms of the type $(p, q)$ on the complex manifold $(M, I)$. By the abuse of notation
we will also denote by the same symbol $\Lambda^{p,q}_{I}(M)$
the space of $C^{\infty}$-sections of this bundle.
The twisted Dolbeault differential operator was introduced in \cite{Ver02} as
\begin{eqnarray*}
\partial_J:=J^{-1}\circ \overline{\partial}\circ J.
\end{eqnarray*}
where
$\overline{\partial}: \Lambda^{p,q}_{I}(M)\rightarrow \Lambda^{p,q+1}_{I}(M)$
is the usual $\overline{\partial}$-differential on differential forms on the complex manifold $(M, I)$.

It is evidently that (see \cite{Ver02})
\begin{proposition}
\begin{eqnarray*}
J: \Lambda^{p,q}_{I}(M)\rightarrow \Lambda^{q, p}_{I}(M),
\end{eqnarray*}
\begin{eqnarray*}
\partial_J: \Lambda^{p,q}_{I}(M)\rightarrow \Lambda^{p+1, q}_{I}(M),
\end{eqnarray*}
\begin{eqnarray*}
\partial \partial_J=-\partial_J\partial.
\end{eqnarray*}
\end{proposition}

\begin{definition}
For each $k=1, 2, \cdot \cdot \cdot, n$. A form $\alpha \in \Lambda^{2k,0}_{I}(M)$ is called q-real if
\begin{eqnarray*}
J\alpha=\overline{\alpha}
\end{eqnarray*}
under the quaternionic conjugation $\mathbb{H}$ \cite{Ver02}.
\end{definition}

The space of q-real $(2k, 0)$ forms on $(M, I)$ will be denoted by $\Lambda^{2k,0}_{I, \mathbb{R}}(M)$.
We have the following lemma \cite{Ale13}.

\begin{lemma}
Let $(M, I, J, K)$ be a hypercomplex manifold.
For $u \in C^{2}(M, \mathbb{R})$, then $\partial\partial_J u \in \Lambda^{2,0}_{I, \mathbb{R}}(M)$
and we call it the quaternionic
Hessian of $u$.
\end{lemma}

\subsection{The fundamental symmetric functions}

In this subsection, we give some basic properties of elementary symmetric functions, which could be found in
\cite{L96, S05}, and establish some key lemmas.

For $\lambda=(\lambda_1, ... , \lambda_n)\in \mathbb{R}^n$,
the $k$-th elementary symmetric function is defined
by
\begin{equation*}
\sigma_k(\lambda) = \sum _{1 \le i_1 < i_2 <\cdots<i_k\leq n}\lambda_{i_1}\lambda_{i_2}\cdots\lambda_{i_k}.
\end{equation*}
We also set $\sigma_0=1$ and denote by $\sigma_k(\lambda \left| i \right.)$ the $k$-th symmetric
function with $\lambda_i = 0$.
Recall that the G{\aa}rding's cone is defined as
\begin{equation*}
\Gamma_k=\{ \lambda  \in \mathbb{R}^n :\sigma _i (\lambda ) > 0, \ \forall \ 1 \le i \le k\}.
\end{equation*}

\begin{proposition}
Let $\lambda=(\lambda_1,\dots,\lambda_n)\in\mathbb{R}^n$ and $k
= 1, \cdots,n$, then
\begin{eqnarray}\label{2.1-1}
\sigma_k(\lambda)=\sigma_k(\lambda|i)+\lambda_i\sigma_{k-1}(\lambda|i),
\quad \forall \,1\leq i\leq n,
\end{eqnarray}
\begin{eqnarray*}\label{2.1-2}
\sum_{i=1}^{n} \lambda_i\sigma_{k-1}(\lambda|i)=k\sigma_{k}(\lambda),
\end{eqnarray*}
\begin{eqnarray*}\label{2.1-3}
\sum_{i=1}^{n}\sigma_{k}(\lambda|i)=(n-k)\sigma_{k}(\lambda).
\end{eqnarray*}
\end{proposition}

The generalized Newton-MacLaurin inequality are as follows, which will be used later.

\begin{proposition}
For $\lambda \in \Gamma_k$ and $n\geq k>l\geq 0$, $ r>s\geq 0$, $k\geq r$, $l\geq s$, we have
\begin{align}\label{NM}
\Bigg[\frac{{\sigma _k(\lambda )}/{C_n^k }}{{\sigma _l (\lambda )}/{C_n^l }}\Bigg]^{\frac{1}{k-l}}
\le \Bigg[\frac{{\sigma _r (\lambda )}/{C_n^r }}{{\sigma _s (\lambda )}/{C_n^s }}\Bigg]^{\frac{1}{r-s}}.
\end{align}
\end{proposition}

\begin{proposition}\label{Ell}
For $\lambda\in\Gamma_{k}$ and $0\leq l<k\leq n$, we have
\begin{eqnarray*}\label{0-1-sum}
\frac{\partial[\frac{\sigma_{k}}{\sigma_{l}}]
(\lambda)}{\partial\lambda_{i}}>0, \quad \forall \ 1\leq i\leq n.
\end{eqnarray*}
\end{proposition}

\begin{proposition}\label{Con}
For any $n\geq k>l\ge 0$,
\begin{eqnarray*}
\bigg[\frac{\sigma_{k}(\lambda)}{\sigma_{l}(\lambda)}\bigg]^{\frac{1}{k-l}}
\end{eqnarray*}
is a concave function in $\Gamma_k$.
\end{proposition}

We recall the G{\aa}rding's inequality (see \cite{CNS85}).

\begin{proposition}
For $1\leq k\leq n$ and $\lambda, \mu \in \Gamma_k$, then we have
\begin{eqnarray*}
\sum_{i=1}^{n}\mu_i\sigma_{k-1}(\lambda|i)\geq k[\sigma_{k}(\mu)]^{\frac{1}{k}}[\sigma_{k}(\lambda)]^{1-\frac{1}{k}}.
\end{eqnarray*}
In particular,
\begin{eqnarray}\label{G-i}
\sum_{i=1}^{n}\mu_i\sigma_{k-1}(\lambda|i)>0.
\end{eqnarray}
\end{proposition}

As a corollary of Proposition \ref{Ell}, we can obtain by using the equality \eqref{2.1-1}.
\begin{proposition}
Assume $\lambda \in \Gamma_k$ and $1\leq l<k\leq n$, then for any $i \in \{1, 2, ..., n\}$ we have
\begin{eqnarray}\label{initial}
\frac{\sigma_{k-1}(\lambda|i)}{\sigma_{l-1}(\lambda|i)}>\frac{\sigma_{k}(\lambda)}{\sigma_{l}(\lambda)}.
\end{eqnarray}
\end{proposition}

\subsection{Hyperhermitian matrices}

In this subsection, we will remind the basic properties of hyperhermitian matrices.
First, we shall also use a version of a determinant defined for
hyperhermitian quaternionic matrices referring for further details, properties.
For a quaternionic $n \times n$-matrix $A \in \mathrm{Mat}_n(\mathbb{H})$ let us denote by $A^{\mathbb{R}}$ the
realization matrix of $A$ which is a real $4n \times 4n$-matrix. Assume $A=A_0+iA_1+jA_2+kA_3$,
where $A_0, A_1, A_2, A_3$ are real $n\times n$ matrices, then (see Page 10 in \cite{Asl96})
\begin{equation}
A^{\mathbb{R}}=\left(
\begin{array}{cccc}
A_0 & -A_1 & -A_2 & -A_3\\
A_1 & A_0 & -A_3 & A_2\\
A_2 & A_3 & A_0 & -A_1\\
A_3 & -A_2 & A_1 & A_0
\end{array}
\right) .
\end{equation}
Clearly, if $A=(a_{ij})$ is hyperhemitian, i.e. $\overline{a_{ij}}=a_{ji}$, $A^{\mathbb{R}}$ is real symmetric.

The following is classical (see \cite{Asl96} for the references).
\begin{theorem}
There exists a polynomial $P$ defined on the space of all hyperhermitian $n\times n$-matrices
such that for any hyperhermitian $n\times n$-matrix
A one has $\mathrm{det}(A^{\mathbb{R}}) = P^4(A)$ and $P(Id)=1$. Furthermore $P$ is homogeneous of degree $n$.
\end{theorem}

\begin{definition}
For a hyperhermitian matrix $A$, its Moore determinant is defined as (see (8) in \cite{Asl96})
\begin{eqnarray*}
\mathrm{det}(A):=P(A) \in \mathbb{R}.
\end{eqnarray*}
The explicit formula for Moore determinant was given by Moore \cite{Moo22} (see also Page 14 in \cite{Asl96}).
\end{definition}

This notation should not cause any confusion with the
usual determinant of real or complex matrices due to part (i) of the next
theorem.
\begin{theorem}\label{A-T-1}
(i) The Moore determinant of any complex hermitian matrix is equal to the usual determinant.

(ii) For any hyperhermitian matrix $A$ and any quaternionic matrix $C$
\begin{eqnarray*}
\mathrm{det}(C^*A C)=\mathrm{det} A \cdot \mathrm{det}(C^*C),
\end{eqnarray*}
where $C^*=\overline{C}^t$.

(iii) For any hyperhermitian matrix $A$ can be symplectically diagonalized.
That is, we can find $C \in GL(n, \mathbb{H})$ with $C^*C=Id$ such that
\begin{eqnarray*}
C^*A C=\mathrm{diag}\{\lambda_1, \lambda_2, ..., \lambda_n\},
\end{eqnarray*}
where $\lambda_i \in \mathbb{R}$ for $1\leq i\leq n$.
\end{theorem}

Let us introduce more notation. Let A be any hyperhermitian $(n \times n)$-matrix. For any
non-empty subset $I\subset \{1,...,n\}$, the minor $M_{I}(A)$ of $A$ which is obtained by deleting the
rows and columns with indexes from the set $I$, is clearly hyperhermitian. For $I=\{1,...,n\}$,
let $\mathrm{det} M_{\{1,...,n\}}=1$. Then, we recall the following important property (see Proposition 1.1.11 in \cite{Ale03}).

\begin{proposition}
For any hyperhermitian $(n \times n)$-matrix $A$ and any diagonal real matrix $T=\mathrm{diag}\{t_1, ..., t_n\}$
\begin{eqnarray*}
\mathrm{det}(A+T)=\sum_{I\subset \{1,...,n\}}\bigg(\prod_{i \in I}t_i\bigg)\cdot \mathrm{det} M_I(A).
\end{eqnarray*}
In particular
\begin{eqnarray}\label{P-1}
\mathrm{det}(A+t \cdot Id)=\sum_{I\subset \{1,...,n\}}t^{|I|}\cdot \mathrm{det} M_I(A).
\end{eqnarray}
where $|I|$ denotes the cardinality of the set $I$.
\end{proposition}

\begin{definition}
Let $A$ be a hyperhermitian $(n \times n)$-matrix, for $k=1, 2, ..., n$, we define
\begin{eqnarray*}
\sigma_k(A)=\sigma_k(\lambda(A))=\sum _{1 \le i_1 < i_2 <\cdots<i_k\leq n}\lambda_{i_1}\lambda_{i_2}\cdots\lambda_{i_k},
\end{eqnarray*}
where $\lambda(A)=(\lambda_1(A), \lambda_2(A), ..., \lambda_n(A))$
are the real eigenvalues of $A$ (see (iii) of Theorem \ref{A-T-1}).
\end{definition}

\begin{remark}\label{R-1}
We can find from Theorem \ref{A-T-1} and the equality \eqref{P-1}
\begin{eqnarray*}
\sigma_k(A)=\frac{1}{(n-k)!}\frac{d^{n-k}}{dt^{n-k}}\mathrm{det}(A+t\cdot Id).
\end{eqnarray*}
Thus, $\sigma_k(A)$ is also the sum of all $(k \times k)$ principle minors of $A$.
\end{remark}

\begin{proposition}\label{A-concave}
Assume $\lambda(A) \in \Gamma_k$, then
\begin{eqnarray*}
\Big[\frac{\sigma_k(A)}{\sigma_l(A)}\Big]^{\frac{1}{k-l}}
\end{eqnarray*}
is concave.
\end{proposition}

\begin{proof}
Let $\lambda_1(A)\leq\lambda_2(A)\leq ...\leq\lambda_n(A)$ be the eigenvalues of $A$. Then,
we know from Theorem \ref{A-T-1} that $\lambda_1(A)\leq\lambda_2(A)\leq ...\leq\lambda_n(A)$
are also the eigenvalues of $A^{\mathbb{R}}$ with multiplicity four.
Applying the same argument for the real symmetric matrix $A^{\mathbb{R}}$ in Page 277 in \cite{CNS85},
we find for $1\leq k\leq n$
\begin{eqnarray*}
\sum_{i=1}^{k}\lambda_i(A)
\end{eqnarray*}
is concave with respect with $A$. Then,
Proposition \ref{Con} implies the conclusion (see Page 277 in \cite{CNS85} for details).
\end{proof}

\begin{proposition}\label{P-2}
Assume $\lambda(A) \in \Gamma_k$, then $\lambda(A|i) \in \Gamma_{k-1}$ for any $1\leq i\leq n$,
where $A|i$ is the matrix obtained by deleting the $i$-row and $i$-column of the matrix $A$.
\end{proposition}

\begin{proof}
Without loss of generality, we only prove $i=n$ and assume
$A|n=\mathrm{diag}\{\mu_1, ..., \mu_{n-1}\}$.
Then, we write
\begin{equation}
A=\left(
\begin{array}{cccc}
A|n & \alpha \\
\overline{\alpha}^t & a_{nn} \\
\end{array}
\right),
\end{equation}
where $\alpha=(\alpha_1, ..., \alpha_{n-1})^t$. Then, we get from Remark \ref{R-1}
\begin{eqnarray*}
\sigma_k(A)&=&\sigma_k(\mu)+a_{nn}\sigma_{k-1}(\mu)
-\sum_{i=1}^{n-1}|\alpha_i|^2\sigma_{k-2}(\mu|i),
\end{eqnarray*}
where $\mu=(\mu_1, ..., \mu_{n-1})$, and
the eigenvalues of $A$ satisfy
\begin{eqnarray*}
\Big(\lambda-a_{nn}-\sum_{i=1}^{n-1}\frac{|\alpha_i|^2}{\lambda-\mu_i}\Big)\prod_{i=1}^{n-1}(\lambda-\mu_i)=0,
\end{eqnarray*}
which implies $\frac{\partial \lambda_l}{\partial a_{nn}}=0$ for $1\leq l\leq n-1$ and
$\frac{\partial \lambda_n}{\partial a_{nn}}=\frac{a_{nn}}{1+\sum_{i=1}^{n-1}\frac{|\alpha_i|^2}{(\lambda-\mu_i)^2}}>0$.
Thus, we have from Proposition \ref{G-i}
\begin{eqnarray*}
\sigma_{k-1}(\mu)=\frac{\partial \sigma_k(A)}{\partial a_{nn}}=
\frac{\partial \sigma_k(A)}{\partial \lambda_l}\frac{\partial \lambda_l}{\partial a_{nn}}>0.
\end{eqnarray*}
\end{proof}

\begin{proposition}\label{A-i}
Assume $\lambda(A) \in \Gamma_k$ and $1\leq l<k\leq n$, then for any $i \in \{1, 2, ..., n\}$ we have
\begin{eqnarray}\label{initial-1}
\frac{\sigma_{k-1}(A|i)}{\sigma_{l-1}(A|i)}>\frac{\sigma_{k}(A)}{\sigma_{l}(A)}.
\end{eqnarray}
\end{proposition}

\begin{proof}
Without loss of generality, we only prove $i=n$ and assume $A|n=\mathrm{diag}\{\mu_1, ..., \mu_{n-1}\}$.
Then, we write
\begin{equation}
A=\left(
\begin{array}{cccc}
A|n & \alpha \\
\overline{\alpha}^t & a_{nn} \\
\end{array}
\right),
\end{equation}
where $\alpha=(\alpha_1, ..., \alpha_{n-1})^t$. Then,
\begin{eqnarray*}
\sigma_k(A)=\sigma_k((\mu, a_{nn}))
-\sum_{i=1}^{n-1}|\alpha_i|^2\sigma_{k-2}(\mu|i),
\end{eqnarray*}
where $\mu=(\mu_1, ..., \mu_{n-1})$.
Applying Proposition \ref{P-2}, the inequalities \eqref{NM} and \eqref{initial}, we get
\begin{eqnarray*}
\frac{\sigma_{k}(A)}{\sigma_{l}(A)}\leq\frac{\sigma_{k}((\mu, a_{nn}))}{\sigma_{l}((\mu, a_{nn}))}.
\end{eqnarray*}
Thus, the inequality \eqref{initial} follows from \eqref{initial}.
\end{proof}

\begin{proposition}
For $1\leq k\leq n$, assume $\lambda=(\lambda_1, ..., \lambda_n) \in \Gamma_k$ with $\lambda_1\leq ...\leq \lambda_n$
and $\lambda(B)=(\mu_1, ..., \mu_n) \in \Gamma_k$ with $\mu_1\leq ...\leq\mu_n$, then we have
\begin{eqnarray}\label{G-ii}
\sum_{i=1}^{n}B_{ii}\sigma_{k-1}(\lambda|i)\geq\sum_{i=1}^{n}\mu_{i}\sigma_{k-1}(\lambda|i)>0.
\end{eqnarray}
\end{proposition}

\begin{proof}
For any hyperhermitian matrix $B$, we can find $C \in GL(n, \mathbb{H})$ with $C^*C=Id$ such that
\begin{eqnarray*}
C^*B C=\mathrm{diag}\{\mu_1, \mu_2, ..., \mu_n\},
\end{eqnarray*}
Then,
\begin{eqnarray*}
\mathrm{diag}\{B_{11}, B_{22}, ..., B_{nn}\}=T\mathrm{diag}\{\mu_1, \mu_2, ..., \mu_n\},
\end{eqnarray*}
where $T=(|c_{ij}|^2)$. Then, the conclusion follows from the same argument Lemma 6.2 in \cite{CNS85}.
\end{proof}

\subsection{Hessian quotient operators on HKT manifolds}

Let $(M, I, J, K, g)$ be a HKT manifold and $\Omega$ be the associated HKT form of the
hyperhermitian metric $g$.
For each $z \in M$, there exists a unit orthonormal
basis $e_1, \overline{e_1}J, ..., e_n, \overline{e_n}J$ of $T^{1, 0}_{z}M$, i.e.
\begin{eqnarray*}\label{2.22}
\Omega(e_i, e_j)=0, \quad \Omega(e_i, \overline{e_j}J)=\delta_{ij}.
\end{eqnarray*}
Let $\{e^i\}$ be the dual basis of $\{e_j\}$, i.e. $e^i(e_j)=\delta_{ij}$. Then,
\begin{eqnarray*}
\Omega=e^1\wedge J^{-1}(\overline{e^1})+...+e^n\wedge J^{-1}(\overline{e^n}).
\end{eqnarray*}
Given $W \in \Lambda^2_{I, \mathbb{R}}(M)$, we can decompose
\begin{eqnarray*}
W=\sum_{i, j=1}^{n}w_{i\overline{j}}e^i\wedge J^{-1}(\overline{e^j}),
\end{eqnarray*}
where $(w_{i\overline{j}})$ is a hyperhermitian $(n \times n)$-matrix, i.e. $\overline{w_{i\overline{j}}}=w_{j\overline{i}}$.

\begin{definition}
For $W \in \Lambda^2_{I, \mathbb{R}}(M)$, we define $\sigma_k(W)$ with
respect to $\Omega$ as
\begin{equation*}
\sigma_k(W)=\sigma_k(\lambda(w_{i\overline{j}})),
\end{equation*}
where $\lambda(w_{i\overline{j}})$ are eigenvalues of the hyperhermitian $(n \times n)$-matrix $(w_{i\overline{j}})$.
The definition of $\sigma_k(W)$ is independent of the choice of the unit orthonormal
basis $e_1, \overline{e_1}J, ..., e_n, \overline{e_n}J$ of $TM$. In fact, $\sigma_k$ can be
defined without the use of the unit orthonormal
basis by (see Proposition 4.5 in \cite{Ale17})
\begin{equation*}
\sigma_{k}(W)=C_{n}^{k}\frac{W^k\wedge \Omega^{n-k}}{\Omega^n},
\end{equation*}
where $C_{n}^{k}=\frac{n!}{(n-k)!k!}$.
We also define the G{\aa}rding's cone on $M$ by
\begin{equation*}
\Gamma_k(M)=\{ W  \in \Lambda^{2, 0}_{I, \mathbb{R}}(M):\sigma _i (W)>0, \ \forall \ 1 \le i \le k\}.
\end{equation*}
\end{definition}

Thus, we can rewrite the equation \eqref{K-eq} as
\begin{eqnarray}\label{Keq-1}
\sigma_{k}(\Omega_u)
=\widetilde{F}\sigma_{l}(\Omega_u),
\end{eqnarray}
where $\widetilde{F}=\frac{C_{n}^{k}}{C_{n}^{l}}e^F$.
Under the unit orthonormal
basis $e_1, J\overline{e_1}, ..., e_n, J\overline{e_n}$ of $T^{1, 0}_{z}M$, it is easy to see
\begin{eqnarray}\label{n-1}
&&W^{i-1}\wedge \Omega^{n-i}\wedge e^l\wedge J^{-1}(\overline{e^l})
\nonumber \\&=&(i-1)!(n-i)!\sigma_{i-1}(W|l)\Omega^n,
\end{eqnarray}
where $W|l$ is the matrix obtained by deleting the $l$-row and $l$-column of the matrix $(w_{i\overline{j}})$.
Then, we can get a local version of the cone condition \eqref{cone}.
\begin{proposition}\label{cone-local}
The cone condition \eqref{cone} is equivalent to
\begin{eqnarray}\label{cone-1}
\sigma_{k-1}(\Omega_0|j)
-\widetilde{F}\sigma_{l-1}(\Omega_0|j)>0, \quad j=1, 2, ..., n,
\end{eqnarray}
where $\widetilde{F}=\frac{C_{n}^{k}}{C_{n}^{l}}e^F$.
\end{proposition}

\begin{proof}
The equality \eqref{cone-1} follows directly from \eqref{n-1} if we
observe that coefficient of $(n-1,0)$ form
$\prod_{j=1, j\neq s}^{n}e^j\wedge J^{-1}(\overline{e^j})$
in $\Omega_{0}^{i-1}\wedge \Omega^{n-i}$ is
\begin{eqnarray*}
(i-1)!(n-i)!\sigma_{i-1}(\Omega_0|s)=\frac{1}{i}\frac{n!}{C_{n}^{i}}\sigma_{i-1}(\Omega_0|s).
\end{eqnarray*}
\end{proof}

\section{$C^{0}$ estimate}

In the section, for convenience, we always assume that $(M^n, I, J, K, g)$ is a
compact HKT manifold, and $\Omega$ is the associated HKT form of the
hyperhermitian metric $g$. Moreover, $\Omega_0$ is a q-real $(2, 0)$ form with $\Omega_0 \in \Gamma_k(M)$ and $\partial\Omega_0=0$, and $u \in C^2(M)$
is a real function with $\Omega_u=\Omega_0+\partial\partial_J u \in \Gamma_k(M)$.

\subsection{Some lemmas}

Since $\Omega_0 \in \Gamma_{k}(M)$, we may assume that there is a uniform constant $\tau>0$ such that
\begin{eqnarray}\label{gama}
\Omega_0-\epsilon\Omega \in \Gamma_{k}(M) \quad \mbox{and} \quad \Omega-\epsilon\Omega_0 \in \Gamma_{k}(M).
\end{eqnarray}

Then, We have the following pointwise inequalities.
\begin{lemma}
(1) For $0\leq t\leq 1$ and $1\leq i\leq k$, it holds
\begin{eqnarray}\label{C0-le-1}
\Omega_{tu}^{i-1}\wedge \omega^{n-i}\geq (1-t)^{i-1}
\epsilon^{i-1}\Omega^{n-1}.
\end{eqnarray}

(2) For $0<t\leq 1$ and $1\leq i\leq k$, it holds
\begin{eqnarray}\label{C0-le-2}
\Omega_{u}^{i}\wedge \Omega^{n-i}\leq \frac{1}{t^i}
\Omega^{i}_{tu}\wedge \Omega^{n-i}.
\end{eqnarray}
Moreover, if $u$ is a solution to the equation \eqref{K-eq} and $\Omega_0$ satisfies the cone condition \eqref{cone},
then there exists a uniform constant $C>0$ such that for $0\leq t\leq 1$
\begin{eqnarray}\label{C0-le-3}
&&k\Omega_{tu}^{k-1}\wedge \Omega^{n-k}-le^F\Omega_{tu}^{l-1}\wedge \Omega^{n-l}\\ \nonumber&>&C(1-t)
\Omega_{tu}^{l-1}\wedge \Omega^{n-l},
\end{eqnarray}
and consequently
\begin{eqnarray}\label{C0-le-4}
&&k\Omega_{tu}^{k-1}\wedge \Omega^{n-k}-le^F\Omega_{tu}^{l-1}\wedge \Omega^{n-l}\\ \nonumber&>&C\epsilon^{l-1}(1-t)^{l}
\Omega^{n-1}.
\end{eqnarray}
\end{lemma}

\begin{proof}
For $1\leq i\leq k$ and $1\leq j\leq n$, we can find from Propositions \ref{A-concave} and \ref{P-2}
\begin{eqnarray}\label{C000-1}
\sigma_{i}^{\frac{1}{i}}(\Omega_{tu})
\geq (1-t)\sigma_{i}^{\frac{1}{i}}(\Omega_0)+t\sigma_{i}^{\frac{1}{i}}(\Omega_u)\geq t\sigma_{i}^{\frac{1}{i}}(\Omega_u),
\end{eqnarray}
\begin{eqnarray}\label{C000-2}
\sigma_{i-1}^{\frac{1}{i-1}}(\Omega_{tu}|j)
&\geq& (1-t)\sigma_{i-1}^{\frac{1}{i-1}}(\Omega_0|j)+t\sigma_{i-1}^{\frac{1}{i-1}}(\Omega_u|j)
\nonumber\\&\geq& (1-t)\sigma_{i-1}^{\frac{1}{i-1}}(\Omega_0|j),
\end{eqnarray}
and
\begin{eqnarray}\label{C000-3}
\sigma_{i-1}^{\frac{1}{i-1}}(\Omega_0|j)
\geq \sigma_{i-1}^{\frac{1}{i-1}}(\Omega_0-\epsilon \Omega|j)+\epsilon\sigma_{i-1}^{\frac{1}{i-1}}(\Omega|j)
\geq \epsilon\sigma_{i-1}^{\frac{1}{i-1}}(\Omega|j).
\end{eqnarray}
Clearly,
\eqref{C0-le-2} is a consequence of \eqref{C000-1}.
Combining \eqref{C000-2} and \eqref{C000-3} gives
\begin{eqnarray*}
\sigma_{i-1}(\Omega_{tu}|j)\geq
(1-t)^{i-1}\epsilon^{i-1}\sigma_{i-1}(\Omega|j),
\end{eqnarray*}
which is just the local form of \eqref{C0-le-1}.
Now we begin to prove \eqref{C0-le-3} and \eqref{C0-le-4}.
Applying Propositions \ref{A-concave} and \ref{P-2}, we know that
\begin{eqnarray*}
f(\alpha|j):=\Big[\frac{\sigma_{k-1}(\alpha|j)}{\sigma_{l-1}(\alpha|j)}\Big]^{\frac{1}{k-l}}
\end{eqnarray*}
is concave for $\alpha$ in $\Gamma_{k}$ for any $1\leq j\leq n$. Thus,
\begin{eqnarray}\label{0-1}
f(\Omega_{tu}|j)\geq (1-t)f(\Omega_{0}|j)
+t f(\Omega_{u}|j).
\end{eqnarray}
Moreover, we have by Proposition \ref{A-i}
\begin{eqnarray}\label{0-2}
f(\Omega_{u}|i)&=&\frac{\sigma_{k-1}(\Omega_{u}|i)}{\sigma_{l-1}(\Omega_{u}|i)}
>\frac{\sigma_{k}(\Omega_{u})}{\sigma_{l}(\Omega_{u})}=\widetilde{F}.
\end{eqnarray}
In addition, since $\Omega_0$ satisfies the cone condition \eqref{cone},
there exists some uniform constant $\delta>0$ which is independent of $z$ such that
\begin{eqnarray}\label{0-3}
f(\Omega_{0}|j)>\widetilde{F}(z)+\delta,
\end{eqnarray}
where we write the cone condition in a local version as  Proposition \ref{cone-local}.
Substituting \eqref{0-2} and \eqref{0-3} into \eqref{0-1}, we get
\begin{eqnarray*}
\frac{\sigma_{k-1}(\Omega_{tu}|j)}{\sigma_{l-1}(\Omega_{tu}|j)}
>(1-t)(\widetilde{F}(z)+\delta)+t\widetilde{F}(z),
\end{eqnarray*}
this is to say
\begin{eqnarray*}
\sigma_{k-1}(\Omega_{tu}|j)
-\widetilde{F}\sigma_{l-1}(\Omega_{tu}|j)
>\delta(1-t)\sigma_{l-1}(\Omega_{tu}|j),
\end{eqnarray*}
which is just the local form of \eqref{C0-le-3}.
Then, \eqref{C0-le-3} follows by substituting \eqref{C0-le-1} into it.
\end{proof}

We recall Lemma 3 in \cite{Sro21}.
\begin{lemma}\label{3}
Let $\Omega_1$ be a strictly positive $(2, 0)$ form, i.e. $\Omega_1(X, \overline{X}J)>0$ for any non zero $(1, 0)$
vector $X$, and $\Omega_2$ a $q$-real $(2, 0)$ form on $M$. For each $z \in M$ there exists a
basis $e_1, \overline{e_1}J, ..., e_n, \overline{e_n}J$ of $T^{1, 0}_{z}M$ such that for $i\neq j$
\begin{eqnarray}\label{2.22}
\Omega_1(e_i, e_j)=\Omega_2(e_i, e_j)=\Omega_1(e_i, \overline{e_j}J)=\Omega_2(e_i, \overline{e_j}J)=0.
\end{eqnarray}
\end{lemma}

Using Lemma \ref{3}, we can get the following inequality which is a generalization of Lemma 2 in \cite{Sro21}.
\begin{lemma}
Let $\alpha$ be some $(1, 0)$ form and $W$ be a q-real $(2, 0)$ form with $W \in \Gamma_i(M)$,
there exists a positive constant $C$ depending on $\alpha$ such that
\begin{eqnarray}\label{Le2}
&&\bigg| \frac{\partial_J u\wedge \alpha\wedge W^{i-1}\wedge \Omega^{n-i}}{\Omega^n}\bigg|
\nonumber\\&\leq&\frac{C}{\delta}\frac{\partial u\wedge \partial_J u\wedge W^{i-1}\wedge \Omega^{n-1-i}}{\Omega^n}
+C\delta \frac{W^{i-1}\wedge \Omega^{n-i+1}}{\Omega^n},
\end{eqnarray}
for some $\delta>0$.
\end{lemma}

\begin{proof}
Using Lemma \ref{3}, for any $z \in M$, we can choose a unit orthonormal
basis $e_1, \overline{e_1}J, ..., e_n, \overline{e_n}J$ of $T^{1, 0}_{z}M$ such that
\begin{eqnarray*}
\Omega=e^1\wedge J^{-1}(\overline{e^1})+...+e^n\wedge J^{-1}(\overline{e^n})
\end{eqnarray*}
and
\begin{eqnarray*}
W=w_1e^1\wedge J^{-1}(\overline{e^1})+...+w_ne^n\wedge J^{-1}(\overline{e^n}).
\end{eqnarray*}
Let us decompose
\begin{eqnarray*}
\alpha=\sum_{j=1}^{n}a_{2j-1}e^j+a_{2j} J^{-1}(\overline{e^j}), \quad
\partial u=\sum_{j=1}^{n}u_{2j-1}e^j+u_{2j} J^{-1}(\overline{e^j}).
\end{eqnarray*}
Thus,
\begin{eqnarray*}
\partial_J u=J^{-1}(\overline{\partial} u)=J^{-1}(\overline{\partial u})
=\sum_{j=1}^{n}\overline{u_{2j-1}}J^{-1}(\overline{e^j})-\overline{u_{2j}}e^j.
\end{eqnarray*}
One can easily check that
\begin{eqnarray*}
W^{i-1}\wedge \Omega^{n-i+1}=\frac{1}{C_{n}^{i-1}}\sigma_{i-1}(W)\Omega^n
\end{eqnarray*}
and
\begin{eqnarray*}
&&\partial u\wedge \partial_J u\wedge W^{i-1}\wedge \Omega^{n-i}
\\&=&\frac{(i-1)!(n-i)!}{n!}\sum_{1\leq j_1<...<j_{i-1}\leq n}\bigg(\sum_{l \notin \{j_1, ..., j_{i-1}\}}|u_{2l-1}|^2+|u_{2l}|^2\bigg)w_{j_1}....w_{j_{i-1}}\Omega^n
\\&=&\frac{(i-1)!(n-i)!}{n!}\sum_{l=1}^{n}(|u_{2l-1}|^2+|u_{2l}|^2)\sigma_{i-1}(W|l)\Omega^n.
\end{eqnarray*}
Moreover,
\begin{eqnarray*}
&&\partial_J u\wedge \alpha\wedge W^{i-1}\wedge \Omega^{n-i}
\\&=&\frac{(i-1)!(n-i)!}{n!}\sum_{1\leq j_1<...<j_{i-1}\leq n}\bigg(\sum_{l \notin \{j_1, ..., j_{i-1}\}}-\overline{u_{2l-1}}a_{2l-1}-\overline{u_{2l}}a_{2l}\bigg)w_{j_1}....w_{j_{i-1}}\Omega^n
\\&=&\frac{(i-1)!(n-i)!}{n!}\sum_{l=1}^{n}\Big(-\overline{u_{2l-1}}a_{2l-1}-\overline{u_{2l}}a_{2l}\Big)\sigma_{i-1}(W|l)\Omega^n.
\end{eqnarray*}
Noting that $\sigma_{i-1}(W|l)>0$ in view of $W \in \Gamma_i(M)$, then
the conclusion follows directly from Cauchy-Schwartz inequality.
\end{proof}

Next, we derive the following two important inequalities.

\begin{lemma}
Suppose that
$\Omega_0$ satisfies
\eqref{gama}, then we have the following inequalities for $1\leq i< k$:
\begin{eqnarray}\label{C00-5}
\epsilon\int_{0}^{a}\Omega_{tu}^{i-1}\wedge \Omega^{n-i+1}dt
\leq\frac{k}{i} \int_{0}^{a}\Omega_{tu}^{k-1}\wedge \Omega^{n-k+1}dt,
\end{eqnarray}
where $a$ is arbitrary positive constant and
\begin{eqnarray}\label{C0-le-5}
&&\epsilon^{k-i}
\int_{0}^{\frac{1}{2}}dt\int_{M}e^{-pu}\partial u\wedge \partial_J u\wedge\Omega_{tu}^{i-1}\wedge \Omega^{n-i}\wedge \overline{\Omega^{n}}\nonumber\\&\leq& \frac{k}{i}\int_{0}^{\frac{1}{2}}dt\int_{M}e^{-pu}\partial u\wedge \partial_J u\wedge\Omega_{tu}^{k-1}\wedge \Omega^{n-k}\wedge \overline{\Omega^{n}}.
\end{eqnarray}
\end{lemma}

\begin{proof}
Using integration by parts, \eqref{gama} and the inequality \eqref{G-ii}
\begin{eqnarray*}
&&\int_{0}^{a}\Omega_{tu}^{i-1}\wedge \Omega^{n-i+1}dt
\\&\geq&\epsilon \int_{0}^{a}\Omega_{tu}^{i-2}\wedge \Omega^{n-i+2}dt
+\frac{1}{i-1}\int_{0}^{a}t\frac{d}{d t}\Omega_{tu}^{i-1}\wedge \Omega^{n-i+1}dt
\\&=&\epsilon \int_{0}^{a}\Omega_{tu}^{i-2}\wedge \Omega^{n-i+2}dt+
\frac{a}{i-1}\Omega_{au}^{i-1}\wedge \Omega^{n-i+1}
-\frac{1}{i-1}\int_{0}^{a}\Omega_{tu}^{i-1}\wedge \Omega^{n-i+1}dt.
\end{eqnarray*}
Hence,
\begin{eqnarray*}
\frac{i}{i-1}\int_{0}^{a}\Omega_{tu}^{i-1}\wedge \Omega^{n-i+1}dt
\geq\epsilon \int_{0}^{a}\Omega_{tu}^{i-2}\wedge \Omega^{n-i+2}dt.
\end{eqnarray*}
Then, the first inequality follows by iteration.
The proof of the second inequality is similar to that of the first one.
Using integration by parts, \eqref{gama} and the inequality \eqref{G-ii}, it yields
\begin{eqnarray*}
&&\int_{0}^{\frac{1}{2}}dt\int_{M}e^{-pu}\partial u\wedge \partial_J u\wedge\Omega_{tu}^{i-1}\wedge \Omega^{n-i}
\wedge \overline{\Omega^{n}}\\&\geq&
\epsilon\int_{0}^{\frac{1}{2}}dt\int_{M}e^{-pu}\partial u\wedge \partial_J u\wedge\Omega_{tu}^{i-2}\wedge \Omega^{n-i+1}\wedge \overline{\Omega^{n}}
\\&&+\frac{1}{i-1}\int_{0}^{\frac{1}{2}}dt\int_{M}e^{-pu}\partial u\wedge \partial_J u\wedge t\frac{d}{dt}(\Omega_{tu}^{i-1}\wedge \Omega^{n-i})\wedge \overline{\Omega^{n}}\\&\geq&
\epsilon\int_{0}^{\frac{1}{2}}dt\int_{M}e^{-pu}\partial u\wedge \partial_J u\wedge\Omega_{tu}^{i-2}\wedge \Omega^{n-i+1}\wedge \overline{\Omega^{n}}
\\&&-\frac{1}{i-1}\int_{0}^{\frac{1}{2}}dt\int_{M}e^{-pu}\partial u\wedge \partial_J u\wedge \Omega_{tu}^{i-1}\wedge \Omega^{n-i}\wedge \overline{\Omega^{n}}.
\end{eqnarray*}
Thus,
\begin{eqnarray*}
&&\frac{i}{i-1}\int_{0}^{\frac{1}{2}}dt\int_{M}e^{-pu}\partial u\wedge \partial_J u\wedge\Omega_{tu}^{i-1}\wedge \Omega^{n-i}\wedge \overline{\Omega^{n}}\\&\geq&
\epsilon\int_{0}^{\frac{1}{2}}dt\int_{M}e^{-pu}\partial u\wedge \partial_J u\wedge\Omega_{tu}^{i-2}\wedge \Omega^{n-i+1}\wedge \overline{\Omega^{n}}
.
\end{eqnarray*}
So, we obtain by iteration
\begin{eqnarray*}
&&\frac{k}{i}\int_{0}^{\frac{1}{2}}dt\int_{M}e^{-pu}\partial u\wedge \partial_J u\wedge\Omega_{tu}^{k-1}\wedge \Omega^{n-k}\wedge \overline{\Omega^{n}}\\&\geq&\epsilon^{k-i}
\int_{0}^{\frac{1}{2}}dt\int_{M}e^{-pu}\partial u\wedge\partial_J u\wedge\Omega_{tu}^{i-1}\wedge \omega^{n-i}\wedge \overline{\Omega^{n}}
,
\end{eqnarray*}
which completes the proof.
\end{proof}

\begin{lemma}
Suppose that
$\Omega_0$ satisfies
\eqref{gama}, then we have the following inequality:
\begin{eqnarray}\label{C00-8}
&&\int_{0}^{\frac{1}{2}}dt\int_{M}e^{-pu}\Omega_{tu}^{k-1}\wedge \Omega^{n-k+1}\wedge\overline{\Omega^n}
\nonumber\\&\leq&C p \int_{0}^{\frac{1}{2}}dt\int_{M}e^{-pu}
\partial u\wedge \partial_J u\wedge\Omega_{tu}^{k-1}\wedge \Omega^{n-k}\wedge\overline{\Omega^n}
\nonumber \\&&+
C\int_{M}e^{-pu}\Omega^{n}\wedge\overline{\Omega^n}.
\end{eqnarray}
\end{lemma}

\begin{proof}
According to the inequality \eqref{C0-le-5}, we have
\begin{eqnarray}\label{C0-5}
&&\int_{0}^{\frac{1}{2}}dt\int_{M}e^{-pu}\Omega_{tu}^{k-1}\wedge \Omega^{n-k+1}\wedge\overline{\Omega^n}
\nonumber \\ \nonumber&=&p(k-1)\int_{0}^{\frac{1}{2}}dt\int_{0}^{t}ds\int_{M}e^{-pu}
\partial u\wedge \partial_J u\wedge\Omega_{su}^{k-2}\wedge \Omega^{n-k+1}\wedge\overline{\Omega^n}
\\ \nonumber&&-(k-1)\int_{0}^{\frac{1}{2}}dt\int_{0}^{t}ds\int_{M}e^{-pu}
\partial_J u\wedge\Omega_{su}^{k-2}\wedge \Omega^{n-k+1}\wedge\partial\overline{\Omega^n}
\nonumber \\&&+
\frac{1}{2}\int_{M}e^{-pu}\Omega_{0}^{k-1}\wedge \Omega^{n-k+1}\wedge\overline{\Omega^n}\nonumber \\&\leq&\frac{p(k-1)}{2}\int_{0}^{\frac{1}{2}}dt\int_{M}e^{-pu}
\partial u\wedge \partial_J u\wedge\Omega_{tu}^{k-2}\wedge \Omega^{n-k+1}\wedge\overline{\Omega^n}\nonumber \\&&+(k-1)\int_{0}^{\frac{1}{2}}dt\int_{M}e^{-pu}
\bigg|\partial_J u\wedge\Omega_{tu}^{k-2}\wedge \Omega^{n-k+1}\wedge\partial\overline{\Omega^n}\bigg|
\nonumber \\&&+
\frac{C}{2}\int_{M}e^{-pu}\Omega^{n}\wedge\overline{\Omega^n}.
\end{eqnarray}
Denote by $\partial \overline{\Omega^n}=\beta \wedge \overline{\Omega^n}$ for some $(1, 0)$ form $\beta$.
Applying the inequality \eqref{Le2}, we find
\begin{eqnarray}\label{DD-1}
&&\bigg| \frac{\partial_J u\wedge \beta\wedge\Omega_{tu}^{k-2}\wedge \Omega^{n-k+1}}{\Omega^n}\bigg|
\nonumber\\&\leq&\frac{p(k-1)}{2}\frac{\partial u\wedge \partial_J u\wedge \Omega_{tu}^{k-2}\wedge \Omega^{n-k+1}}{\Omega^n}
+\frac{C}{p} \frac{\Omega_{tu}^{k-2}\wedge \Omega^{n-k+2}}{\Omega^n}.
\end{eqnarray}
Substituting \eqref{DD-1} into \eqref{C0-5} and applying \eqref{C0-le-5} give
\begin{eqnarray}\label{C0-5}
&&\int_{0}^{\frac{1}{2}}dt\int_{M}e^{-pu}\Omega_{tu}^{k-1}\wedge \Omega^{n-k+1}\wedge\overline{\Omega^n}
\nonumber\\&\leq&p(k-1)\int_{0}^{\frac{1}{2}}dt\int_{M}e^{-pu}
\partial u\wedge \partial_J u\wedge\Omega_{tu}^{k-2}\wedge \Omega^{n-k+1}\wedge\overline{\Omega^n}\nonumber \\&&+\frac{C}{p}\int_{0}^{\frac{1}{2}}dt\int_{M}e^{-pu}
\Omega_{tu}^{k-2}\wedge \Omega^{n-k+2}\wedge\overline{\Omega^n}
\nonumber \\&&+
\frac{C}{2}\int_{M}e^{-pu}\Omega^{n}\wedge\overline{\Omega^n}\nonumber\\&\leq&p C\int_{0}^{\frac{1}{2}}dt\int_{M}e^{-pu}
\partial u\wedge \partial_J u\wedge\Omega_{tu}^{k-1}\wedge \Omega^{n-k}\wedge\overline{\Omega^n}\nonumber \\&&+\frac{C}{p}\int_{0}^{\frac{1}{2}}dt\int_{M}e^{-pu}
\Omega_{tu}^{k-1}\wedge \Omega^{n-k+1}\wedge\overline{\Omega^n}
\nonumber \\&&+
\frac{C}{2}\int_{M}e^{-pu}\Omega^{n}\wedge\overline{\Omega^n}.
\end{eqnarray}
where we use the inequality \eqref{C00-5} for $a=\frac{1}{2}$ to get the last inequality. Choosing $p$ sufficiently large, we find
\begin{eqnarray*}
&&\int_{0}^{\frac{1}{2}}dt\int_{M}e^{-pu}\Omega_{tu}^{k-1}\wedge \Omega^{n-k+1}\wedge\overline{\Omega^n}
\nonumber\\&\leq&C p \int_{0}^{\frac{1}{2}}dt\int_{M}e^{-pu}
\partial u\wedge \partial_J u\wedge\Omega_{tu}^{k-1}\wedge \Omega^{n-k}\wedge\overline{\Omega^n}
\nonumber \\&&+
C\int_{M}e^{-pu}\Omega^{n}\wedge\overline{\Omega^n}.
\end{eqnarray*}
\end{proof}

\subsection{The proof of $C^{0}$ estimate}

According to the argument of Page 13-14 in \cite{Sro21}, in order to prove Theorem \ref{Main}, it suffices to show
the following Cherrier type inequality.

\begin{lemma}
Under the assumptions in Theorem \ref{Main}, we have for $p$ large enough
\begin{eqnarray*}
\int_{M}|\partial e^{-\frac{p}{2}u}|^2(\Omega\wedge\overline{\Omega})^n\leq C p\int_{M}e^{-pu}(\Omega\wedge\overline{\Omega})^n.
\end{eqnarray*}
\end{lemma}

\begin{proof}

Since
\begin{eqnarray*}
&&(\Omega_{u}^k\wedge \Omega^{n-k}-\Omega_{0}^k\wedge \Omega^{n-k})
-e^{F}(\Omega_{u}^l\wedge \Omega^{n-l}-\Omega_{0}^l\wedge \Omega^{n-l})
\\&=&\int_{0}^{1}\partial\partial_J u\wedge\Big(k\Omega_{tu}^{k-1}\wedge \Omega^{n-k}
-le^{F}\Omega_{tu}^{l-1}\wedge \Omega^{n-l}\Big)dt,
\end{eqnarray*}
we have
\begin{eqnarray}\label{C0-1}
&&\int_{M}e^{-pu}\Big[(\Omega_{u}^{k}\wedge \Omega^{n-k}-\Omega_{0}^k\wedge \omega^{n-k})
\\&&-e^{F}(\Omega_{u}^{l}\wedge \Omega^{n-l}-\Omega_{0}^l\wedge \Omega^{n-l})\Big]\wedge \overline{\Omega^n}
\nonumber\\ \nonumber
&=&p\int_{0}^{1}dt\int_{M}e^{-pu}\partial u\wedge \partial_J u\wedge\Big(k\Omega_{tu}^{k-1}\wedge \Omega^{n-k}
-e^{F}l\Omega_{tu}^{l-1}\wedge \Omega^{n-l}\Big)\wedge \overline{\Omega^n}\\ \nonumber&&
+l\int_{0}^{1}dt\int_{M}e^{-pu}\partial_J u\wedge \partial e^F \wedge\Omega_{tu}^{l-1}\wedge \Omega^{n-l}\wedge \overline{\Omega^n}
\nonumber\\ \nonumber
&&-\int_{0}^{1}dt\int_{M}e^{-pu}\partial_J u\wedge\Big(k\Omega_{tu}^{k-1}\wedge \Omega^{n-k}
-e^{F}l\Omega_{tu}^{l-1}\wedge \Omega^{n-l}\Big)\wedge \partial\overline{\Omega^n}.
\end{eqnarray}
Denote by $\partial \overline{\Omega^n}=\beta \wedge \overline{\Omega^n}$ for some $(1, 0)$ form $\beta$.
Applying the inequality \eqref{Le2}, we find
\begin{eqnarray}\label{D-1}
&&\bigg| \frac{\partial_J u\wedge \beta\wedge \Big(k\Omega_{tu}^{k-1}\wedge \Omega^{n-k}
-e^{F}l\Omega_{tu}^{l-1}\wedge \Omega^{n-l}\Big)}{\Omega^n}\bigg|
\nonumber\\&\leq&\frac{p}{2}\frac{\partial u\wedge \partial_J u\wedge \Big(k\Omega_{tu}^{k-1}\wedge \Omega^{n-k}
-e^{F}l\Omega_{tu}^{l-1}\wedge \Omega^{n-l}\Big)}{\Omega^n}
\nonumber\\&&+\frac{C}{p}\frac{\Omega_{tu}^{k-1}\wedge \Omega^{n-k+1}+\Omega_{tu}^{l-1}\wedge \Omega^{n-l+1}}{\Omega^n}.
\end{eqnarray}
and
\begin{eqnarray}\label{D-2}
&&\bigg| \frac{\partial_J u\wedge \partial e^F\wedge \wedge\Omega_{tu}^{l-2}\wedge \Omega^{n-l+1}}{\Omega^n}\bigg|
\nonumber\\&\leq&\frac{C}{\delta}\frac{\partial u\wedge \partial_J u\wedge \Omega_{tu}^{l-2}\wedge \Omega^{n-l+1}}{\Omega^n}
+C\delta \frac{\Omega_{tu}^{l-2}\wedge \Omega^{n-l+2}}{\Omega^n}.
\end{eqnarray}
Then, plugging \eqref{D-1} and \eqref{D-2} into \eqref{C0-1}, and using \eqref{C00-5}, there is
\begin{eqnarray}\label{C0-2}
&&\int_{M}e^{-pu}\Big[(\Omega_{u}^{k}\wedge \Omega^{n-k}-\Omega_{0}^k\wedge \omega^{n-k})
\nonumber\\&&-e^{F}(\Omega_{u}^{l}\wedge \Omega^{n-l}-\Omega_{0}^l\wedge \Omega^{n-l})\Big]\wedge \overline{\Omega^n}
\nonumber\\ \nonumber
&\geq&\frac{p}{2}\int_{0}^{1}dt\int_{M}e^{-pu}\partial u\wedge \partial_J u\wedge\Big(k\Omega_{tu}^{k-1}\wedge \Omega^{n-k}
-e^{F}\Omega_{tu}^{l-1}\wedge \Omega^{n-l}\Big)\wedge \overline{\Omega^n}\\ \nonumber&&
-\frac{C}{\delta} \int_{0}^{1}dt\int_{M}e^{-pu}\partial u \wedge\partial_J u \wedge\Omega_{tu}^{l-1}\wedge \Omega^{n-l}\wedge \overline{\Omega^n}
\nonumber\\
&&-(\frac{1}{p}+\delta)C\int_{0}^{1}dt\int_{M}e^{-pu}\Omega_{tu}^{k-1}\wedge \Omega^{n-k+1}
\wedge\overline{\Omega^n},
\end{eqnarray}
Furthermore, applying the inequalities \eqref{C000-2} and \eqref{C0-le-5}, we find
\begin{eqnarray}\label{C00-9}
&&\int_{0}^{1}dt\int_{M}e^{-pu}\partial u \wedge\partial_J u\wedge\Omega_{tu}^{l-1}\wedge \Omega^{n-l}\wedge\overline{\Omega^n}
\nonumber\\&\leq&2^{l-1}\int_{0}^{1}dt\int_{M}e^{-pu}\partial u \wedge\partial_J u\wedge\Omega_{\frac{tu}{2}}^{l-1}\wedge \Omega^{n-l}\wedge\overline{\Omega^n}
\nonumber\\ \nonumber&\leq&2^{l-1}\int_{0}^{\frac{1}{2}}dt\int_{M}e^{-pu}\partial u \wedge\partial_J u\wedge\Omega_{tu}^{l-1}\wedge \Omega^{n-l}\wedge\overline{\Omega^n}\nonumber\\&\leq&C\int_{0}^{\frac{1}{2}}dt\int_{M}e^{-pu}\partial u \wedge\partial_J u\wedge\Omega_{tu}^{k-1}\wedge \Omega^{n-k}\wedge\overline{\Omega^n}.
\end{eqnarray}
and
\begin{eqnarray}\label{C00-2}
&&\int_{0}^{1}dt\int_{M}e^{-pu}\Omega_{tu}^{k-1}\wedge \Omega^{n-k+1}\wedge\overline{\Omega^n}
\nonumber\\&\leq&2^{k-1}\int_{0}^{1}dt\int_{M}e^{-pu}\Omega_{\frac{tu}{2}}^{k-1}\wedge \Omega^{n-k+1}\wedge\overline{\Omega^n}
\nonumber\\&\leq&2^{k-1}\int_{0}^{\frac{1}{2}}dt\int_{M}e^{-pu}\Omega_{tu}^{k-1}\wedge \Omega^{n-k+1}\wedge\overline{\Omega^n}\nonumber\\&\leq&C p \int_{0}^{\frac{1}{2}}dt\int_{M}e^{-pu}
\partial u\wedge \partial_J u\wedge\Omega_{tu}^{k-1}\wedge \Omega^{n-k}\wedge\overline{\Omega^n}
\nonumber \\&&+
C\int_{M}e^{-pu}\Omega^{n}\wedge\overline{\Omega^n}.
\end{eqnarray}
Taking the inequalities \eqref{C00-9} and \eqref{C00-2} into \eqref{C0-2}, we have
\begin{eqnarray}\label{C2-2}
&&\int_{M}e^{-pu}\Big[(\Omega_{u}^{k}\wedge \Omega^{n-k}-\Omega_{0}^k\wedge \omega^{n-k})
\\ \nonumber&&-e^{F}(\Omega_{u}^{l}\wedge \Omega^{n-l}-\Omega_{0}^l\wedge \Omega^{n-l})\Big]\wedge \overline{\Omega^n}
\\ \nonumber
&\geq&\frac{p}{2}\int_{0}^{1}dt\int_{M}e^{-pu}\partial u\wedge \partial_J u\wedge\Big(k\Omega_{tu}^{k-1}\wedge \Omega^{n-k}
-e^{F}\Omega_{tu}^{l-1}\wedge \Omega^{n-l}\Big)\wedge \overline{\Omega^n}\\ \nonumber&&
-(\frac{1}{\delta}+1+p\delta)C\int_{0}^{1}dt\int_{M}e^{-pu}\partial u \wedge\partial_J u \wedge\Omega_{tu}^{k-1}\wedge \Omega^{n-k}\wedge \overline{\Omega^n}
\nonumber\\ \nonumber
&&-(\frac{1}{p}+\delta)C\int_{M}e^{-pu}\Omega^{n}\wedge\overline{\Omega^n}.
\end{eqnarray}

To cancel the second term on the right side of \eqref{C2-2}, we will use a part of the first term on the right side of \eqref{C2-2}. In
details, we can get the following positive term for $0\leq t\leq \frac{1}{2}$ from the inequality \eqref{C0-le-3}
\begin{eqnarray}\label{C0-6}
&&pe^{-pu}\partial u\wedge \partial_J u\wedge\Big(k\Omega_{tu}^{k-1}\wedge \Omega^{n-k}
-le^F\Omega_{tu}^{l-1}\wedge \Omega^{n-l}\Big)\nonumber\\&\geq&
Cpe^{-pu}\partial u\wedge \partial_J u\wedge\Omega_{tu}^{k-1}\wedge \Omega^{n-k}.
\end{eqnarray}
Thus, we first choose $\delta$ sufficiently small, and then choose $p$ sufficiently large,
 the integral of the term \eqref{C0-6} on $M$ can
kill the second term on the right side of \eqref{C2-2}. Then, \eqref{C2-2} becomes
\begin{eqnarray*}\label{C1-2}
&&\int_{M}e^{-pu}\Big[(\Omega_{u}^{k}\wedge \Omega^{n-k}-\Omega_{0}^k\wedge \omega^{n-k})
-e^{F}(\Omega_{u}^{l}\wedge \Omega^{n-l}-\Omega_{0}^l\wedge \Omega^{n-l})\Big]\wedge \overline{\Omega^n}
\nonumber\\ \nonumber
&\geq&\frac{p}{4}\int_{0}^{1}dt\int_{M}e^{-pu}\partial u\wedge \partial_J u\wedge\Big(k\Omega_{tu}^{k-1}\wedge \Omega^{n-k}
-e^{F}\Omega_{tu}^{l-1}\wedge \Omega^{n-l}\Big)\wedge \overline{\Omega^n}\\&&-C\int_{M}e^{-pu}\Omega^{n}\wedge\overline{\Omega^n}
\nonumber\\ \nonumber
&\geq&Cp\int_{0}^{1}dt\int_{M}e^{-pu}\partial u\wedge \partial_J u\wedge\Omega^{n-1}\wedge \overline{\Omega^n}\\&&-C\int_{M}e^{-pu}\Omega^{n}\wedge\overline{\Omega^n}.
\end{eqnarray*}
where we use \eqref{C0-le-4} to get the last inequality.
Notice that in view of the equation \eqref{K-eq}
\begin{eqnarray*}
&&\int_{M}e^{-pu}\Big[(\Omega_{u}^{k}\wedge \Omega^{n-k}-\Omega_{0}^k\wedge \Omega^{n-k})
-e^F(\Omega_{u}^{l}\wedge \Omega^{n-l}-\Omega_{0}^l\wedge \Omega^{n-l})\Big]\wedge \overline{\Omega^n}
\\&=& \int_{M}e^{-pu}\Big(-\Omega_{0}^k\wedge \Omega^{n-k}
+e^F\Omega_{0}^l\wedge \Omega^{n-l}\Big)\wedge \overline{\Omega^n}\nonumber \\ \nonumber&\leq&C\int_{M}e^{-pu}\Omega^{n}\wedge \overline{\Omega^n}.
\end{eqnarray*}
Thus,
\begin{eqnarray*}\label{C1-2}
C\int_{M}e^{-pu}\Omega^{n}\wedge\overline{\Omega^n}
\geq p\int_{M}e^{-pu}\partial u\wedge \partial_J u\wedge\Omega^{n-1}
\wedge \overline{\Omega^n}.
\end{eqnarray*}
So, our proof is completed.
\end{proof}

 \bibliographystyle{siam}

\end{document}